\documentclass[a4paper,11pt, reqno]{amsart}
\usepackage{amsmath,amsfonts,amssymb}
\usepackage{amsthm}
\usepackage{graphics,graphicx,epsf}
\usepackage[usenames]{color}
\usepackage{color}
\usepackage{hyperref} 
\usepackage{verbatim}

\newtheorem{theorem}{Theorem}[section]

\newtheorem{proposition}[theorem]{Proposition}
\newtheorem{lemma}[theorem]{Lemma}

\newtheorem{remark}[theorem]{Remark}

\numberwithin{equation}{section}

\title[A  dissipative logarithmic type evolution equation]{A dissipative logarithmic type evolution equation of second order in time}

\author[F. L. Oliveira]{F\'abio L. de Oliveira$^1$}\thanks{$^1$Research partially supported by
CAPES \# 88887.661986/2022-00, Brazil.}
\address[F. L. Oliveira]{Universidade Federal da Para\'{\i}ba, Departamento de Matem\'atica, Cidade Universit\'{a}ria, 58051-900 Jo\~{a}o Pessoa PB, Brazil.}
\email{fabiolimaoliveira99@gmail.com}

\author[D. G. Santos]{Diego G. dos Santos}
\address[D. G. Santos]{Universidade Federal da Para\'{\i}ba, Departamento de Matem\'atica, Cidade Universit\'{a}ria, 58051-900 Jo\~{a}o Pessoa PB, Brazil.}
\email{diego.gomes2@academico.ufpb.br}

\author[M. J. M. Silva]{Maria J. M. Silva}
\address[M. J. M. Silva]{Universidade Federal da Para\'{\i}ba, Departamento de Matem\'atica, Cidade Universit\'{a}ria, 58051-900 Jo\~{a}o Pessoa PB, Brazil.}
\email{mjms@academico.ufpb.br}

\author[D. J. C. Silva]{Dennys J. C. Silva}
\address[D. J. C. Silva]{Departamento de Matem\'atica, Universidade Federal da Para\'iba,
Cidade Universit\'{a}ria-Campus I, 58051-900 Jo\~{a}o Pessoa PB, Brazil.}
\email{dennys.costasilva@gmail.com}

\date{\today}

\begin{document}

 \maketitle

\begin{abstract}
In this paper, we introduce a logarithmic-type second-order model with a non-local logarithmic damping mechanism in $\mathbb{R}^N$. We present a motivation with a spectral approach to consider the equation, we consider the Cauchy problem associated with the model. More precisely, we study the asymptotic behavior of solutions as t goes to infinity in $L^2$-sense; namely, we prove results on the asymptotic profile and optimal decay of solutions as time goes to infinity in $L^2$-sense. 
\end{abstract}

\vskip .1 in \noindent {\it Mathematical Subject Classification 2020:} 35L05, 35B40, 35C20, 35S05.
\newline {\it Key words and phrases:} asymptotic profile; Fourier transform; L2-decay; Logarithmic damping; optimal estimates

\tableofcontents

\section{Introduction}

In this paper, we are particularly interested in the asymptotic profile and optimal decay of solutions  as time goes to infinity, in $L^2-$sense, of the following evolution equation of second order in time,, which naturally arises from the spectral properties of the damped free wave operator in $\mathbb{R}^N$ 
\begin{equation}\label{Pro1}
\partial_t^2u+\dfrac{1}{4}\log^2(-\Delta+I)u+\dfrac{\pi^2}{4}u+\log(-\Delta+I)\partial_tu=0,
\end{equation}
subject to initial conditions
\begin{equation}\label{Pro2}
u(x,0)=u_0(x),\ \partial_tu(x,0)=u_1(x),\ x\in\mathbb{R}^N.
\end{equation}
where $N\geq3$, and $L=\log(-\Delta+I)$ denotes  the unbounded linear operator
\[
L:D(L)\subset L^2(\mathbb{R}^N)\to L^2(\mathbb{R}^N)
\]
defined by
\[
D(L)=\left\{f\in L^2(\mathbb{R}^N); \int_{\mathbb{R}^N}\log^2(1+|\xi|^2)|\mathcal{F}(f)(\xi)|^2d\xi<+\infty \right\}
\]
and for each $f\in D(L)$
\[
(Lf)(x)=\mathcal{F}^{-1}_{\xi\to x}(\log(1+|\xi|^2) \mathcal{F}(f)(\xi))(x),
\]
where $\mathcal{F}_{x\to\xi}(\mathcal(f)(\xi))$ denotes the Fourier transform of $f\in L^2(\mathbb{R}^N)$ at $\xi$, and $\mathcal{F}^{-1}_{x\to\xi}(\mathcal(f)(\xi))(x)$ denotes the inverse Fourier transform of $f$ at $x$.

We also observe that the linear operator $\log(-\Delta+I)$ in \eqref{Pro1}  can be characterized by
\[
\log(-\Delta+I)u(x):=\int_{\mathbb{R}^N}\dfrac{u(x)-u(y)}{|y-x|^N}\omega(|y-x|)dy,
\]
with $\omega(r)=c_Nr^{\frac{N}{2}}K_{\frac{N}{2}}(r)$, $c_N$ is a constant given by
\[
c_N=\dfrac{\Gamma(\frac{N}{2})}{\pi^{\frac{N}{2}}},
\]
where $\Gamma$ denotes the Gamma function,  and $K_{\frac{N}{2}}$ is a  modified Bessel function of second kind and index $\frac{N}{2}$. In particular, $\omega(r)=c_N+b_Nr+o(r)$ as $r\to0^+$, for more details see e.g. \cite{ Feulefack, SiSoVo}.

\subsection{Motivation}

Heuristically, if  we consider the following linear wave equation with potential
\begin{equation}\label{Mot1}
\partial_t^2u-\Delta u+u=0,
\end{equation}
in $\mathbb{R}^N$ subject to initial conditions \eqref{Pro2}, where $N\in\mathbb{N}$, $u_0\in H^1(\mathbb{R}^N)$ and $u_1\in L^2(\mathbb{R}^N)$. Then, it is well known that the  problem \eqref{Mot1}-\eqref{Pro2} can be rewrite as the following Cauchy problem on $H^1(\mathbb{R}^N)\times L^2(\mathbb{R}^N)$
\begin{equation}\label{Mot3}
\dfrac{d}{dt}\begin{bmatrix} u\\ \partial_tu\end{bmatrix}+\begin{bmatrix} 0 & -I\\  -\Delta+I & 0\end{bmatrix}\begin{bmatrix} u\\ \partial_tu\end{bmatrix}=\begin{bmatrix} 0\\ 0\end{bmatrix},\ t>0,
\end{equation}
subject to initial condition
\begin{equation}\label{Mot4}
\begin{bmatrix} u\\ \partial_tu\end{bmatrix}(0)=\begin{bmatrix} u_0\\ u_1\end{bmatrix}.
\end{equation}
The unbounded linear operator $\left[\begin{smallmatrix} 0 & -I\\  -\Delta+I & 0\end{smallmatrix}\right]$  on the Hilbert space $H^1(\mathbb{R}^N)\times L^2(\mathbb{R}^N)$ with domain $H^2(\mathbb{R}^N)\times H^1(\mathbb{R}^N)$ is the infinitesimal generator of a strongly continuous semigroup of bounded linear operators on $H^1(\mathbb{R}^N)\times L^2(\mathbb{R}^N)$, see e.g. \cite[Section 7.4, Chapter 7]{P}.

As mentioned by Chen and Russel in \cite{ChenRussel} `Perhaps the most notable disadvantage associated with conservative systems is that they do not occur in nature. There are always dissipative mechanisms within the system, causing the energy to decrease during any positive time interval'. This has inspired us to consider the dissipative counterpart of \eqref{Mot3}-\eqref{Mot4}.

Motived by the analysis performed in \cite{B1} we can consider the Logarithmic counterpart of \eqref{Mot3}-\eqref{Mot4} given by the dissipative system
\begin{equation}\label{Mot5}
\dfrac{d}{dt}\begin{bmatrix} u\\ \partial_tu\end{bmatrix}+\left(\log\begin{bmatrix} 0 & -I\\  -\Delta+I & 0\end{bmatrix}\right)\begin{bmatrix} u\\ \partial_tu\end{bmatrix}=\begin{bmatrix} 0\\ 0\end{bmatrix},\ t>0,
\end{equation}
subject to initial condition \eqref{Mot4}, where 
\[
\log\begin{bmatrix} 0 & -I\\  -\Delta+I & 0\end{bmatrix}=- \left(-\log\begin{bmatrix} 0 & -I\\  -\Delta+I & 0\end{bmatrix}\right)
\]
 denotes the logarithm of the operator $\left[\begin{smallmatrix} 0 & -I\\  -\Delta+I & 0\end{smallmatrix}\right]$ in the sense of \cite[Page 152]{A}; namely, $-\log\left[\begin{smallmatrix} 0 & -I\\  -\Delta+I & 0\end{smallmatrix}\right]$ denotes the  infinitesimal generator of a strongly continuous semigroup of bounded linear operators $\{\left[\begin{smallmatrix} 0 & -I\\  -\Delta+I & 0\end{smallmatrix}\right]^{-1};t\geqslant0\}$ on $H^1(\mathbb{R}^N)\times L^2(\mathbb{R}^N)$, 
\[
-\log\begin{bmatrix} 0 & -I\\  -\Delta+I & 0\end{bmatrix}=\lim_{\alpha\to0^+}\dfrac{1}{\alpha}\left(\begin{bmatrix} 0 & -I\\  -\Delta+I & 0\end{bmatrix}^{-\alpha}-\begin{bmatrix} I & 0\\ 0 & I\end{bmatrix}\right)
\]
and by the Balakrishnan integral formula
\[
-\log\begin{bmatrix} 0 & -I\\  -\Delta+I & 0\end{bmatrix}=\begin{bmatrix} \frac{1}{2}\log(-\Delta+I ) & -\frac{\pi}{2}A^{-\frac12}\\  \frac{\pi}{2}A^{\frac12} & \frac{1}{2}\log(-\Delta+I )\end{bmatrix}
\]
that is
\[
\log\begin{bmatrix} 0 & -I\\  -\Delta+I & 0\end{bmatrix}=\begin{bmatrix} -\frac{1}{2}\log(-\Delta+I ) & \frac{\pi}{2}A^{-\frac12}\\  -\frac{\pi}{2}A^{\frac12} & -\frac{1}{2}\log(-\Delta+I )\end{bmatrix}
\]
is the unbounded linear operator on the Hilbert space $H^1(\mathbb{R}^N)\times L^2(\mathbb{R}^N)$ with domain 
\[
[D(\log(-\Delta+I )\cap H^{\frac12}(\mathbb{R}^N)]\times D(\log(-\Delta+I ))
\]
present in \eqref{Mot5}, and the equation associated with \eqref{Mot5} to be solved by $u$ can be rewrite as the equation  \eqref{Pro1}, see the ideas in \cite{B1}.

Dissipative logarithmic type evolution equations such as \eqref{Pro1} have been treated in the literature in recent years in different configurations, see e.g. \cite{B1} and \cite{CIP}, see also \cite{MAF3}.  Namely, in \cite{CIP} the dissipative logarithmic type evolution equation
\[
\partial_t^2u+\log(-\Delta+I)u+\log(-\Delta+I)\partial_tu=0
\]
was considered in $\mathbb{R}^N$, where $\Delta$ is the usual Laplace operator defined in $H^2(\mathbb{R}^N)$. The authors consider the Cauchy problem for this new model in $\mathbb{R}^N$, and study the asymptotic profile and optimal decay and/or blow-up rates of solutions as time goes to infinity in $L^2-$sense.

Here, since the linear operator $L$ is non-negative and self-adjoint in $L^2(\mathbb{R}^N)$, the square root
\[
L^{\frac12}:D(L^{\frac12})\subset  L^2(\mathbb{R}^N)\to L^2(\mathbb{R}^N)
\]
can be defined and is also nonnegative and self-adjoint with its domain
\[
D(L^{\frac12})=\left\{f\in L^2(\mathbb{R}^N); \int_{\mathbb{R}^N}\log(1+|\xi|^2)|\mathcal{F}(f)(\xi)|^2d\xi<+\infty \right\}.
\]

Note that $D(L^{\frac12})$ becomes Hilbert space with its graph norm
\[
\|f\|_{D(L^{\frac12})}=\left(\|f\|_{L^2(\mathbb{R}^N)}^2+\|L^{\frac12}f\|_{L^2(\mathbb{R}^N)}^2\right)^{\frac12}
\]
It is easy to check that
\[
H^s(\mathbb{R}^N) \hookrightarrow
 D(L^{\frac12}) \hookrightarrow
 L^2(\mathbb{R}^N)
\]
for $s>0$, see e.g. \cite{CIP}.

Now, for the moment, we choose $ D(L)\times L^2(\mathbb{R}^N)$ endowed with the norm given by
\[
\left\|\begin{bmatrix} u\\ \partial_tu\end{bmatrix}\right\|_{D(L)\times L^2(\mathbb{R}^N)}= \dfrac{1}{4}\|L  u\|_{L^2(\mathbb{R}^N)}^2+\dfrac{\pi^2}{4}\| u\|_{L^2(\mathbb{R}^N)}^2+\|v\|_{L^2(\mathbb{R}^N)}^2
\]
 and the initial data $\left[\begin{smallmatrix} u_0\\  u_1\end{smallmatrix}\right]$ as follows
\[
u_0\in D(L),\ u_1\in L^2(\mathbb{R}^N).
\]
Concerning the existence of a unique solution to problem \eqref{Pro1}-\eqref{Pro2}, by a similar argument to \cite[Proposition 2.1]{RIke} (see also \cite{CIP}) based on the Lumer-Phillips Theorem, one can find that problem \eqref{Pro1}-\eqref{Pro2} with initial data $\left[\begin{smallmatrix} u_0\\  u_1\end{smallmatrix}\right]\in D(L)\times L^2(\mathbb{R}^N)$ has a unique mild solution
\[
u\in C([0,+\infty); D(L))\cap C^1([0,+\infty); L^2(\mathbb{R}^N)),
\]
and the associated energy identity holds
\[
E_u(t)+\int_0^t\|L^{\frac12}\partial_tu(\cdot,s)\|_{L^2(\mathbb{R}^N)}^2ds=E_u(0),\ t>0
\]
where
\[
E_u(t)=\dfrac{1}{2}\Big(\|\partial_tu(\cdot,t)\|_{L^2(\mathbb{R}^N)}^2+\dfrac{1}{4}\|L  u(\cdot,t)\|_{L^2(\mathbb{R}^N)}^2+\dfrac{\pi^2}{4}\| u(\cdot,t)\|_{L^2(\mathbb{R}^N)}^2\Big)
\]
see below for more details. 

We aim to investigate the behavior of solutions for the problem \eqref{Pro1}-\eqref{Pro2}, to find an asymptotic profile of solutions in the $L^2$ framework to the problem \eqref{Pro1}-\eqref{Pro2}, and to apply it to investigate the rate of decay of solutions in terms of the $L^2$-norm, as well as in \cite{CIP}.

\subsection{Organization of the paper}

This paper is organized as follows. In Section \ref{Sec2} we present the problem \eqref{Pro1}-\eqref{Pro2} in the Fourier space and study the decay of solutions as $t$ goes to infinity in the $L^2-$sense. In Section \ref{FinalSection} we present a result on the asymptotic profile for solution of  \eqref{Pro1}-\eqref{Pro2} $L^2-$sense. Finally, in Section \ref{sec4} we present a result on the optimal decay rate of $L^2-$norm for solution of  \eqref{Pro1}-\eqref{Pro2} $L^2-$sense.

\section{Preliminaries}\label{Sec2}

We need to consider the equivalent problem to \eqref{Pro1}-\eqref{Pro2} in the Fourier space which is
\begin{equation}\label{Pro3}
\partial_t^2\hat u(\xi,t)+\dfrac{1}{4}\log^2(1+|\xi|^2)\hat u(\xi,t)+\dfrac{\pi^2}{4}\hat u(\xi,t)+\log(1+|\xi|^2)\partial_t \hat u(\xi,t)=0,\  \xi\in\mathbb{R}^N,\ t>0,
\end{equation}
 subject to initial conditions
\begin{equation}\label{Pro4}
\hat u(\xi,0)=\hat u_0(\xi),\ \partial_t \hat u(\xi,0)=\hat u_1(\xi),\ \xi\in\mathbb{R}^N.
\end{equation}

The characteristic roots $\lambda_\pm$ for the characteristic polynomial associated with \eqref{Pro3} are given by
\[
\lambda_\pm=\dfrac{1}{2}\Big[-\log(1+|\xi|^2)\pm i\dfrac{\pi}{2}\Big],\ \xi\in\mathbb{R}^N,
\]
that is, $\lambda_\pm$ are are complex-valued for all $\xi\in\mathbb{R}^N$.

It is easy to check that the solution in the Fourier space can be explicitly expressed as
\[
\begin{split}
\hat u(\xi,t)&=e^{-\frac{t\log(1+|\xi|^2)}{2}}\cos\Big(\dfrac{\pi t}{4}\Big)\hat u_0(\xi)+e^{-\frac{t\log(1+|\xi|^2)}{2}}\dfrac{2\log(1+|\xi|^2)}{\pi}\sin\Big(\dfrac{\pi t}{4}\Big)\hat u_0(\xi)\\
&+e^{-\frac{t\log(1+|\xi|^2)}{2}}\dfrac{4}{\pi}\sin\Big(\dfrac{\pi t}{4}\Big)\hat u_1(\xi).
\end{split}
\]

First, multiplying the equation  \eqref{Pro1} by $\partial_t\hat u$ one can get the following point wise energy identity
\begin{equation}\label{Ene1}
\dfrac{d}{dt}E_0(\xi,t)+\log(1+|\xi|^2)|\partial_t\hat u(\xi,t)|^2=0,
\end{equation}
where
\begin{equation}\label{Ene2}
E_0(\xi,t)=\dfrac{1}{2}|\partial_t \hat u(\xi,t)|^2+\dfrac{1}{8}\log^2(1+|\xi|^2)|\hat u(\xi,t)|^2+\dfrac{\pi^2}{8}|\hat u(\xi,t)|^2,
\end{equation}
for $\xi\in\mathbb{R}^N$ and $t>0$, is the total density energy of the system \eqref{Pro3}-\eqref{Pro4}. Note from \eqref{Ene1} that $E_0(\xi,t)$ is a decreasing function of $t$ for each $\xi$ along the solutions of  \eqref{Pro1}.

Second, multiplying the equation \eqref{Pro3} by $\rho(\xi)\bar{ \hat u}$  we obtain the identity
\[
\rho(\xi)\dfrac{d}{dt}(\partial_t \hat u \bar{ \hat u})-\rho(\xi)|\partial_t \hat u|^2+\dfrac{1}{4}\log^2(1+|\xi|^2)|\hat u|^2+\dfrac{\pi^2}{4}|\hat u|^2+\log(1+|\xi|^2) \rho(\xi)\dfrac{d}{dt}\dfrac{|\hat u|^2}{2}=0,
\]
for $ \xi\in\mathbb{R}^N,\ t>0$. 

Taking the real part of the last identity we arrive at
\[
\dfrac{d}{dt}\Big[ \rho(\xi)(\partial_t \hat u\bar{ \hat u})+\rho(\xi)\log(1+|\xi|^2) \dfrac{|\hat u|^2}{2}\Big] +\dfrac{1}{4}\log^2(1+|\xi|^2)|\hat u|^2+\dfrac{\pi^2}{4}|\hat u|^2= \rho(\xi)|\partial_t \hat u|^2
\]
which holds for $ \xi\in\mathbb{R}^N,\ t>0$. 

To proceed further we need to define the following functions on $\mathbb{R}^N\times (0,+\infty)$
\begin{equation}\label{Ene2DFV}
E(\xi,t)=E_0(\xi,t)+  \rho(\xi)(\partial_t \hat u \bar{ \hat u})+\rho(\xi)\log(1+|\xi|^2) \dfrac{|\hat u|^2}{2};
\end{equation}
\begin{equation}\label{Ene2CVBH}
F(\xi,t)= \log(1+|\xi|^2)|\partial_t \hat u|^2 + \dfrac{1}{4}[\log^2(1+|\xi|^2)+\pi^2]|\hat u|^2
\end{equation}
and
\begin{equation}\label{Ene2NHYG}
R(\xi,t)=\rho(\xi)|\partial_t \hat u |^2.
\end{equation}

Note that  we have the following identity
\[
\dfrac{d}{dt}E(\xi,t)+F(\xi,t)=R(\xi,t),
\]
for  $\xi \in \mathbb{R}^N$ and $t>0$.

The following technical results are well-known by   \cite[Lemma 2.1 and Lemma 2.2]{CIP}.

\begin{lemma}\label{l2.1}
Let $p>-1$. If 
\[
I_p(t) = \int_0^1 (1 + r^2)^{-t}r^{p}dr,
\]
then, there exist constants $c_1>0$ and $c_2>0$ such that 
\[
c_1t^{-\frac{p + 1}{2}}\leq I_p(t) \leq c_2 t^{-\frac{p + 1}{2}}, \ \ \mbox{for} \ \ t \gg 1,
\]
\end{lemma}

\begin{lemma}\label{l2.2}
Let $p>-1$. If
\[
J_p(t) = \int_1^{\infty}(1 + r^2)^{-t}r^{p}dr,
\]
then, there exist constants $c_1>0$ and $c_2>0$ such that 
\[
c_1\frac{2^{-t}}{t - 1}\leq J_p(t) \leq c_2\frac{2^{-t}}{t - 1}, \ \ \mbox{for} \ \ t \gg 1.
\]
\end{lemma}

We also consider the following weighted Lebesgue spaces
\[
L^{1, \kappa}(\mathbb{R}^N) := \left\{f \in L^1(\mathbb{R}^N) \ | \ \|f\|_{1, \kappa} := \int_{\mathbb{R}^N }(1 + |x|^{\kappa})|f(x)|dx < +\infty \right\}.
\]

Next,  we shall obtain optimal estimates of the total energy of the following Fourier transformed equation together with initial data of the original system  \eqref{Pro1}-\eqref{Pro2}. To do so we employ the so-called energy method in the Fourier space. We define the following function of $\xi$ on $\mathbb{R}^N$ to be used in the research of results on the asymptotic behavior of solutions for of \eqref{Pro1}-\eqref{Pro2} as $t$ goes to infinity in $L^2-$sense.
\begin{equation}\label{Ene3}
	\rho(\xi)= \begin{cases} \dfrac{1}{4}\log(1+|\xi|^{2}),&\,\,\mbox{if}\,\, |\xi|\leqslant\sqrt{e^{\frac{\pi}{\sqrt{3}}}-1} \\
		\dfrac{1}{16}\dfrac{\log^{2}(1+|\xi|^{2})+\pi^{2}}{\log(1+|\xi|^{2})},& \,\,\mbox{if}\,\,|\xi|>\sqrt{e^{\frac{\pi}{\sqrt{3}}}-1}.\end{cases}
\end{equation}

\medskip

Before continuing our argument, we need the next result.

 \begin{lemma}\label{Lemma2.1}
 	The function $\rho(\xi)$ defined in (\ref{Ene3}) satisfies the estimates
 	$$\rho(\xi)\leqslant\frac{\pi}{4\sqrt{3}}$$
 	
 	\noindent for $|\xi|\leqslant\sqrt{e^{\frac{\pi}{\sqrt{3}}}-1}$. Moreover,
 	$$\rho^2(\xi)\leqslant \frac{1}{16}\log^2(1+|\xi|^2)$$
 	
 	\noindent for all $\xi \in \mathbb{R}^N$.
 \end{lemma}
 
 \begin{proof}
 	In fact, for $|\xi|\leqslant\sqrt{e^{\frac{\pi}{\sqrt{3}}}-1}$ we have 
 	\[
  \rho(\xi)=\frac{1}{4}\log(1+|\xi|^2)\leqslant \frac{1}{4}\log(e^{\frac{\pi}{\sqrt{3}}})=\frac{\pi}{4\sqrt{3}}.
    \]
In addition, is immediate from the definition of $\rho$ that
 	\[
  \rho^2(\xi)\leqslant \frac{1}{16}\log^2(1+|\xi|^2),
    \]
for all $\xi \in \mathbb{R}^N$.
\qed 
 \end{proof} 
 
\begin{proposition} \label{lema 2.2}
The following inequalities 
 	$$\frac{1}{2}E_{0}(\xi, t) \leqslant E(\xi, t) \leqslant \frac{9}{4}E_{0}(\xi, t),$$
it holds for  $\xi \in \mathbb{R}^N$ and $t>0$. 
 \end{proposition}
 
 \begin{proof} Using the inequality 
 \[
 \rho(\xi)\mbox{Re}(\partial_t \hat u \bar{ \hat u})\geq - \frac{|\partial_t \hat u|^2}{4} - \rho^2(\xi)|\hat u|^2
 \]
  and Proposition \ref{lema 2.2} we have
 	\begin{align*} 
 		E(\xi, t) &= E_0(\xi, t) + \rho(\xi)(\partial_t \hat u \bar{ \hat u}) + \rho(\xi)\log(1+|\xi|^2)\frac{|\hat u|^2}{2} \\
 		& \geq E_0(\xi, t) + \rho(\xi)(\partial_t \hat u \bar{ \hat u}) \\
 		& \geq E_0(\xi, t) - \frac{|\partial_t \hat u|^2}{4} - \rho^2(\xi)|\hat u|^2 \\
 		& = \dfrac{1}{2}|\partial_t \hat u|^2+\dfrac{1}{8}\log^2(1+|\xi|^2)|\hat u|^2+\dfrac{\pi^2}{8}|\hat u|^2 - \frac{|\partial_t \hat u|^2}{4} - \rho^2(\xi)|\hat u|^2 \\
 		& = \frac{|\partial_t \hat u|^2}{4} + \left[\dfrac{1}{8}\log^2(1+|\xi|^2)+\dfrac{\pi^2}{8}- \rho^2(\xi)\right]|\hat u|^2	\\	
 		& \geq \frac{|\partial_t \hat u|^2}{4} + \left[\dfrac{1}{8}\log^2(1+|\xi|^2)+\dfrac{\pi^2}{8}- \left(\dfrac{1}{16}\log^2(1+|\xi|^2)+\dfrac{\pi^2}{16}\right)\right]|\hat u|^2 \\
 		& = \frac{|\partial_t \hat u|^2}{4} + \left[\dfrac{1}{16}\log^2(1+|\xi|^2)+\dfrac{\pi^2}{16}\right]|\hat u|^2 \\
 		& = \dfrac{1}{2}\left(\dfrac{1}{2}|\partial_t \hat u|^2+\dfrac{1}{8}\log^2(1+|\xi|^2)|\hat u|^2+\dfrac{\pi^2}{8}|\hat u|^2\right). 
 	\end{align*}
Consequently
\[
E(\xi, t)\geqslant \dfrac{1}{2}E_0(\xi, t).
\]
for $\xi \in \mathbb{R}^N$ and $t>0$.

On the other hand, using Lemma \ref{Lemma2.1} we have
 	\begin{align*}
 		E(\xi, t) &= E_0(\xi, t) + \rho(\xi)(\partial_t \hat u \bar{\hat u}) + \rho(\xi)\log(1+|\xi|^2)\frac{|\hat u|^2}{2} \\
 		& \leqslant E_0(\xi, t) +  \dfrac{|\partial_t \hat u|^2}{2}+ \dfrac{\rho^2(\xi)}{2}|\hat u|^2 + \dfrac{\rho(\xi)}{2}\log(1+|\xi|^2)|\hat u|^2 \\
 		& \leqslant E_0(\xi, t) +  \dfrac{|\partial_t \hat u|^2}{2}+ \dfrac{1}{32}\log^2(1+|\xi|^2)|\hat u|^2 + \dfrac{1}{8}\log^2(1+|\xi|^2)|\hat u|^2 \\
 		& \leqslant E_0(\xi, t) +  \dfrac{|\partial_t \hat u|^2}{2} + \dfrac{5}{32}\log^2(1+|\xi|^2)|\hat u|^2 \\
 		& = |\partial_t \hat u|^2 + \left[\dfrac{1}{8}\log^2(1+|\xi|^2)+\dfrac{\pi^2}{8} + \dfrac{5}{32}\log^2(1+|\xi|^2)\right]|\hat u|^2 \\
 		& \leqslant |\partial_t \hat u|^2 + \left[\dfrac{9}{32}\log^2(1+|\xi|^2)+\dfrac{9\pi^2}{32}\right]|\hat u|^2 \\
 		& \leqslant |\partial_t \hat u|^2 + \dfrac{9}{4}\left[\dfrac{1}{8}\log^2(1+|\xi|^2)+\dfrac{\pi^2}{32}\right]|\hat u|^2 + \dfrac{|\partial_t \hat u|^2}{8} \\
 		& = \dfrac{9}{4}\left[\dfrac{|\partial_t \hat u|^2}{2} +\left(\dfrac{1}{8}\log^2(1+|\xi|^2)+\dfrac{\pi^2}{32}\right)|\hat u|^2\right] \\
 		& = \dfrac{9}{4} E_0(\xi, t), \\
 	\end{align*}

 Consequently
\[
E(\xi, t)\leqslant \dfrac{9}{4} E_0(\xi, t),
\]
for $\xi \in \mathbb{R}^N$ and $t>0$.
\qed 
 \end{proof}

\medskip

For technical reasons, in order to obtain the asymptotic behavior of the solution of  problem \eqref{Pro1}-\eqref{Pro2}, we now consider the following function of $\xi$  defined on $\mathbb{R}^N$ by 

\begin{equation}\label{Ene4}
	\varphi(\xi)= \begin{cases} \dfrac{2}{3}\log(1+|\xi|^{2}),&\,\,\mbox{if}\,\, |\xi|\leqslant\sqrt{e^{\frac{4}{3}}-1} \\[0.4cm]
		\dfrac{8}{9},& \,\,\mbox{if}\,\,|\xi|> \sqrt{e^{\frac{4}{3}}-1}.\end{cases}
\end{equation}

\begin{proposition}\label{Lemma3_3}
The following inequality
\[
\dfrac{d}{dt}E(\xi,t)+\varphi(\xi) E(\xi,t)\leqslant 0,
\]
holds for $\xi \in \mathbb{R}^N$ and $t>0$. 
\end{proposition} 
 
 \begin{proof}
 Note that by \eqref{Ene2}, \eqref{Ene2DFV}, \eqref{Ene2CVBH}, \eqref{Ene2NHYG}, Lemma \ref{Lemma2.1} and Proposition \ref{lema 2.2} we have
 \[
 \begin{split}
 &\dfrac{d}{dt}E(\xi,t)+\varphi(\xi) E(\xi,t)\\
 &=  R(\xi,t)-F(\xi,t)+\varphi(\xi) E(\xi,t) \\
 &\leqslant  R(\xi,t)-F(\xi,t)+\frac{9}{4}\varphi(\xi)E_{0}(\xi, t) \\
 &\leqslant \left [ \dfrac{3}{8} (3\varphi(\xi) - 2\log (1 + |\xi|^2))\right]|\partial_t \hat u(\xi , t)|^2 + \left [\frac{\log ^2 (1 + |\xi|^2) + \pi^2}{4} \right]\left [\frac{9\varphi(\xi) - 8}{8} \right]| \hat u (\xi , t)|^2\\
 &\leqslant 0,
 \end{split}
 \]
for  $\xi \in \mathbb{R}^N$ and $t>0$.
 \end{proof}

\medskip
Now we may note that Proposition \ref{Lemma3_3} implies 
\[
E(\xi,t)\leqslant E(\xi,0) e^{-\varphi(\xi)t}
\]
Combining the last estimate with Proposition \ref{lema 2.2} we arrive at the important proposition.
\[
E_0(\xi,t)\leqslant\dfrac{9}{2} E_0(\xi,0) e^{-\varphi(\xi)t}
\]
for  $\xi \in \mathbb{R}^N$ and $t>0$.
 
 That is, using the definition of $E(\xi,t)$ we have obtained the important point wise estimates in the Fourier space.

\begin{proposition}\label{Prop3.1Refdasrjhk}
The following inequality
\begin{equation}\label{eq prop3.1}
\begin{split}
&|\partial_t \hat u(\xi,t)|^2+\dfrac{1}{4}\log^2(1+|\xi|^2)| \hat u (\xi,t)|^2+\dfrac{\pi^2}{4}| \hat u (\xi,t)|^2\\
&\leqslant \dfrac{9}{2}\Big(|\hat u_1 (\xi)|^2+\dfrac{1}{4}\log^2(1+|\xi|^2)|\hat u_0 (\xi)|^2+\dfrac{\pi^2}{4}| \hat u_0 (\xi)|^2\Big) e^{-\varphi(\xi)t},
\end{split}
\end{equation}
 it holds for $\xi\in\mathbb{R}^N$ and  $t>0$, and
\begin{equation}\label{eq2prop3.1}
|\hat u (\xi,t)|^2 \leqslant 18\left( \dfrac{1}{\log ^2(1 + |\xi|^2)+\pi^2}|\hat u_1 (\xi)|^2 + \dfrac{1}{4}| \hat u_0 (\xi)|^2 \right) e^{-\varphi(\xi)t},
\end{equation}
for $\xi\in\mathbb{R}^N$ and  $t>0$.
\end{proposition}

Now we investigate the decay rate of the total energy $E_u$ and $L^2-$norm of the solution itself under the $L^1-$framework on the initial data.

\begin{theorem}\label{asaSFT}
Let $u(x,t)$ be the solution to problem \eqref{Pro1}-\eqref{Pro2} with initial data
\[
\begin{bmatrix}u_0\\ u_1 \end{bmatrix}\in (D(L)\cap  L^1(\mathbb{R}^N))\times (L^2(\mathbb{R}^N)\cap L^1(\mathbb{R}^N))
\]
Then, the total energy of this system satisfies for $t \gg 0$
\[
\begin{split}
&\|\partial_t u(\cdot,t)\|^2_{L^2(\mathbb{R}^N)}+\dfrac{1}{4}\|Lu(\cdot,t)\|^2_{L^2(\mathbb{R}^N)}+\dfrac{\pi^2}{4}\| u(\cdot,t)\|_{L^2(\mathbb{R}^N)}^2\\ &\leqslant C\left(\| u_1\|^{2}_{L^1(\mathbb{R}^N)} t^{-\frac{N}{2}} +\| u_0\|^{2}_{L^1(\mathbb{R}^N)} t^{-\frac{N+4}{2}} + \|  u_0\|^{2}_{L^1(\mathbb{R}^N)}  \pi^2 t^{-\frac{N}{2}}  \right)\\& + C\left( 2^{-\frac{2}{3}t} \| u_1\|^{2}_{L^2(\mathbb{R}^N)} +\left(\dfrac{16 + 9\pi^2}{36} \right) 2^{-\frac{2}{3}t}\|  u_0\|_{L^2(\mathbb{R}^N)}^2  \right) + C e^{-\frac{8}{9}t} E_u (0),
\end{split}
\]where $C$ is a positive constant depending only on $N$.
\end{theorem}

\begin{proof}
Thanks to \eqref{eq prop3.1} and   the Plancherel's theorem (see e.g. \cite[Theorem 1, p. 183]{Evans}), we have 
\[
\begin{split}
&\|\partial_tu(\cdot,t)\|^2_{L^2(\mathbb{R}^N)}+\dfrac{1}{4}\|Lu(\cdot,t)\|^2_{L^2(\mathbb{R}^N)}+\dfrac{\pi^2}{4}\| u(\cdot,t)\|_{L^2(\mathbb{R}^N)}^2\\
&=\|\partial_t \hat u(\cdot,t)\|^2_{L^2(\mathbb{R}^N)}+\dfrac{1}{4}\|\log(1+|\cdot|^2) w(\cdot,t)\|^2_{L^2(\mathbb{R}^N)}+\dfrac{\pi^2}{4}\| w(\cdot,t)\|_{L^2(\mathbb{R}^N)}^2\\
&=\int_{\mathbb{R}^N}\Big(|\partial_t \hat u(\xi,t)|^2+\dfrac{1}{4}\log^2(1+|\xi|^2)| \hat u (\xi,t)|^2+\dfrac{\pi^2}{4}| \hat u (\xi,t)|^2\Big)d\xi\\
&\leqslant \dfrac{9}{2} \int_{\mathbb{R}^N}\Big(|\hat u_1 (\xi)|^2+\dfrac{1}{4}\log^2(1+|\xi|^2)|\hat u_0 (\xi)|^2+\dfrac{\pi^2}{4}| \hat u_0 (\xi)|^2\Big) e^{-\varphi(\xi)t} d\xi\\
&= \dfrac{9}{2}\left[ \int_{|\xi|\leqslant 1}|\hat u_1 (\xi)|^2 e^{-\varphi(\xi)t} d\xi+\dfrac{1}{4}\int_{|\xi|\leqslant 1}\Big(\log^2(1+|\xi|^2)+ \pi^2\Big)| \hat u_0 (\xi)|^2 e^{-\varphi(\xi)t} d\xi\right]\\
&+ \dfrac{9}{2}\left[ \int_{1<|\xi|\leqslant k}|\hat u_1 (\xi)|^2 e^{-\varphi(\xi)t} d\xi+\dfrac{1}{4}\int_{1<|\xi|\leqslant k}\Big(\log^2(1+|\xi|^2)+ \pi^2\Big)| \hat u_0 (\xi)|^2 e^{-\varphi(\xi)t} d\xi\right]\\
&+ \dfrac{9}{2}\left[ \int_{|\xi|>k}|\hat u_1 (\xi)|^2 e^{-\varphi(\xi)t} d\xi+\dfrac{1}{4}\int_{|\xi|>k}\Big(\log^2(1+|\xi|^2)+ \pi^2\Big)| \hat u_0 (\xi)|^2 e^{-\varphi(\xi)t} d\xi\right]\\
&= \dfrac{9}{2}\left(A_1 + A_2 + A_3 \right),
\end{split}
\]
where $k=\sqrt{e^{\frac{4}{3}} -1}$ and $A_i$ ($i =1, 2, 3$) are the integrals at low, medium and high frequencies, respectively.

\medskip

\noindent $\bullet$ Estimate on the zone $|\xi| \leqslant 1$: 

In this zone, we have $\varphi(\xi)= \dfrac{2}{3}\log(1+|\xi|^2 ) $. Assuming the initial data, $u_0$ and $u_1$, are in $L^{1}(\mathbb{R}^N)$, then $\hat{u}_0 ,\hat{u}_1 \in L^{\infty}(\mathbb{R}^N)$, and thus:
\[
\|w_0 \|_{\infty} \leqslant \| u_0 \|_{1}\quad\mbox{and}\quad \|w_1 \|_{\infty} \leqslant \| u_1 \|_{1}.
\]
\[
\begin{split}
A_1 &= \int_{|\xi|\leqslant 1}|\hat u_1 (\xi)|^2 e^{-\frac{2}{3}\log(1+|\xi|^2)t} d\xi+\dfrac{1}{4}\int_{|\xi|\leqslant 1}\Big(\log^2(1+|\xi|^2)+ \pi^2\Big)| \hat u_0 (\xi)|^2 e^{-\frac{2}{3}\log(1+|\xi|^2)t} d\xi\\
&= \int_{|\xi|\leqslant 1}|\hat u_1 (\xi)|^2 (1+|\xi|^2)^{-\frac{2}{3}t} d\xi+\dfrac{1}{4}\int_{|\xi|\leqslant 1}\Big(\log^2(1+|\xi|^2)+ \pi^2\Big)|\hat u_0 (\xi)|^2 (1+|\xi|^2)^{-\frac{2}{3}t} d\xi\\
&\leqslant \|w_1 \|^{2}_\infty \int_{|\xi|\leqslant 1} (1+|\xi|^2)^{-\frac{2}{3}t} d\xi+\dfrac{\| w_0\|^{2}_\infty }{4}\int_{|\xi|\leqslant 1}\Big(\log^2(1+|\xi|^2)+ \pi^2\Big) (1+|\xi|^2)^{-\frac{2}{3}t} d\xi\\
&\leqslant \|u_1 \|^{2}_1 \int_{|\xi|\leqslant 1} (1+|\xi|^2)^{-\frac{2}{3}t} d\xi+\dfrac{\|  u_0\|^{2}_1 }{4}\int_{|\xi|\leqslant 1}\Big(\log^2(1+|\xi|^2)+ \pi^2\Big) (1+|\xi|^2)^{-\frac{2}{3}t} d\xi\\
&= \| u_1\|^{2}_1 \omega_N \int_{0}^{1} (1+r^2 )^{-\frac{2}{3}t}r^{N-1} dr +\dfrac{\| u_0\|^{2}_1 }{4} \omega_N \int_{0}^{1} \Big(\log^2(1+r^2 )+ \pi^2 \Big) (1+r^2 )^{-\frac{2}{3}t} r^{N-1}dr\\ 
&\leqslant \| u_1\|^{2}_1 \omega_N \int_{0}^{1} (1+r^2 )^{-\frac{2}{3}t}r^{N-1} dr +\dfrac{\| u_0\|^{2}_1 }{4} \omega_N \int_{0}^{1} r^{N-1}( r^4 + \pi^2 ) (1+r^2 )^{-\frac{2}{3}t} dr,
\end{split}
\]
where $\omega_N := \int_{|\omega| = 1} d\omega$ is the surface area of the $N$-dimensional unit ball.

Due to the fact that $\log(1+r^2 )\leqslant r^2$ for all $r\geq 0$. Moreover, from Lemma \ref{l2.1}, we may obtain for $t\gg 1$
\[
\begin{split}
A_1 &\leqslant \| u_1\|^{2}_1 \omega_N \int_{0}^{1} (1+r^2 )^{-\frac{2}{3}t}r^{N-1} dr +\dfrac{\|  u_0\|^{2}_1 }{4} \omega_N \int_{0}^{1} r^{N+3}(1+r^2 )^{-\frac{2}{3}t} dr \\&+ \dfrac{\| u_0\|^{2}_1 }{4} \omega_N \pi^2\int_{0}^{1} r^{N-1}(1+r^2 )^{-\frac{2}{3}t} dr\\ 
&\leqslant \| u_1\|^{2}_1 \omega_N I_{N-1}\left( \frac{2}{3}t \right) +\dfrac{\| u_0\|^{2}_1 }{4} \omega_N I_{N+3}\left( \frac{2}{3}t \right) + \dfrac{\|  u_0\|^{2}_1 }{4} \omega_N \pi^2 I_{N-1}\left(\frac{2}{3}t\right)\\ 
&= \|u_1\|^{2}_1 \omega_N \left( \frac{2}{3}t \right)^{-\frac{N}{2}} +\dfrac{\| u_0\|^{2}_1 }{4} \omega_N \left( \frac{2}{3} t\right)^{-\frac{N+4}{2}} + \dfrac{\|  u_0\|^{2}_1 }{4} \omega_N \pi^2 \left(\frac{2}{3}t\right)^{-\frac{N}{2}}\\ 
&\leqslant C_N \left(\| u_1\|^{2}_1 t^{-\frac{N}{2}} +\|  u_0\|^{2}_1 t^{-\frac{N+4}{2}} + \|  u_0\|^{2}_1  \pi^2 t^{-\frac{N}{2}}  \right),
\end{split}
\]
where $C_N$ is a positive constant depending only on $N$.

\medskip

\noindent $\bullet$ Estimate on the zone $1\leqslant|\xi| \leqslant \sqrt{e^{\frac{4}{3}} -1}=k$:

 In this intermediate region, we also observe that $\varphi(\xi)=\dfrac{2}{3}\log(1+|\xi|^2 )$. Consequently, we can approximate the $\log(1 + |\xi|^2)$ as follows:
$$\log(2)\leqslant \log(1 + |\xi|^2)\leqslant \dfrac{4}{3}.$$
Thus,
\[
\begin{split}
A_2 &=\int_{1\leqslant|\xi|\leqslant k}|\hat u_1 (\xi)|^2 e^{-\frac{2}{3}\log(1+|\xi|^2 )t} d\xi+\dfrac{1}{4}\int_{1\leqslant|\xi|\leqslant k}\left(\log^2(1+|\xi|^2)+ \pi^2\right)| \hat u_0 (\xi)|^2 e^{-\frac{2}{3}\log(1+|\xi|^2 )t} d\xi\\
&\leqslant\int_{1\leqslant|\xi|\leqslant k}|\hat u_1 (\xi)|^2 e^{-\frac{2}{3}\log(2 )t} d\xi+\dfrac{1}{4}\int_{1\leqslant|\xi|\leqslant k}\left(\dfrac{16}{9}+ \pi^2\right)| \hat u_0 (\xi)|^2 e^{-\frac{2}{3}\log(2 ) t} d\xi\\
&\leqslant   2^{-\frac{2}{3}t} \|w_1\|^{2}_2 +\left(\dfrac{4}{9}+ \dfrac{\pi^2}{4} \right) 2^{-\frac{2}{3}t}\| w_0\|_{2}^2 \\
&=   2^{-\frac{2}{3}t} \|u_1\|^{2}_2 +\left(\dfrac{16 + 9\pi^2}{36} \right) 2^{-\frac{2}{3}t}\|  u_0\|_{2}^2 .
\end{split}
\]

\medskip

\noindent $\bullet$ Estimate on the zone $|\xi|\geq \sqrt{e^{\frac{4}{3}} -1}=k$:

In this high frequency region, we have $\varphi(\xi)=\dfrac{8}{9}$. Therefore,
\[
\begin{split}
A_3 &=\int_{\xi|\geq k}|\hat u_1 (\xi)|^2 e^{-\frac{8}{9}t} d\xi+\dfrac{1}{4}\int_{|\xi|\geq k}\left(\log^2(1+|\xi|^2)+ \pi^2\right)| \hat u_0 (\xi)|^2 e^{-\frac{8}{9}t} d\xi\\
&\leqslant e^{-\frac{8}{9}t} \|w_1\|^{2}_{L^2(\mathbb{R}^N)} + \dfrac{e^{-\frac{8}{9}t}}{4}\int_{\mathbb{R}^N}\left(\log^2(1+|\xi|^2)+ \pi^2\right)| \hat u_0 (\xi)|^2d\xi\\
&=e^{-\frac{8}{9}t}\left( \|\partial_t \hat u(\cdot,0)\|^2_{L^2(\mathbb{R}^N)}+\dfrac{1}{4}\|\log(1+|\cdot|^2) w(\cdot,0)\|^2_{L^2(\mathbb{R}^N)}+\dfrac{\pi^2}{4}\| w(\cdot,0)\|_{L^2(\mathbb{R}^N)}^2 \right)\\
&=e^{-\frac{8}{9}t}\left( \| u_1\|^2_{L^2(\mathbb{R}^N)}+\dfrac{1}{4}\|\log(1+|\cdot|^2) u_0\|^2_{L^2(\mathbb{R}^N)}+\dfrac{\pi^2}{4}\| u_0 \|_{L^2(\mathbb{R}^N)}^2 \right)\\
&=2e^{-\frac{8}{9}t} E_u (0),\quad t>0.
\end{split}
\]
By combining the estimates for $A_1$, $A_2$, $A_3$ the proof is now complete.
    \qed
\end{proof}

\begin{remark}
The Theorem \ref{asaSFT} states that the total energy of the system decays as $t^{-\frac{N}{2}}$; that is,
\[
E_u(t)\leq C_{1,n}(\|u_0\|^2_{L^1(\mathbb{R}^N)}+\|u_0\|^2_{L^2(\mathbb{R}^N)}+\|u_1\|^2_{L^1(\mathbb{R}^N)}+\|u_1\|^2_{L^2(\mathbb{R}^N)}+E_u(0))t^{-\frac{N}{2}},\quad t \gg 1,
\]
with a constant $C_{1,n}>0$ that depends only on $n$.
\end{remark}

\begin{proposition}
Let $u(x,t)$ be the solution to problem \eqref{Pro1}-\eqref{Pro2} with initial data
\[
 u_0 , u_1 \in L^2(\mathbb{R}^N) \cap L^1(\mathbb{R}^N).
\]
Then, for $t\gg 1$,
\[
\|u(\cdot , t) \|_{L^2(\mathbb{R}^N)} \leqslant C_N \left(\|u_0\|_{L^2(\mathbb{R}^N)} +\|u_0\|_{L^1(\mathbb{R}^N)} + \|u_1\|_{L^1(\mathbb{R}^N)}  \right)t^{-\frac{N}{2}},
\]
where $C_N> 0$ is a constant depending only on $ N$.
\end{proposition}

\begin{proof}
By integrating the inequality (\ref{eq2prop3.1}) on $\mathbb{R}^N$ and using again the Plancherel's theorem (see e.g. \cite[Theorem 1, p. 183]{Evans}) we obtain
\[
\begin{split}
\|u(t,\cdot)\|^2_{L^2(\mathbb{R}^N)}  &\leqslant 18\int_{\mathbb{R}^N}\left( \dfrac{1}{\log ^2(1 + |\xi|^2)+\pi^2}|\hat u_1 (\xi)|^2 + \dfrac{1}{4}| \hat u_0 (\xi)|^2 \right) e^{-\varphi(\xi)t} d\xi\\
&= 18\int_{|\xi|\leqslant 1}\left( \dfrac{1}{\log ^2(1 + |\xi|^2)+\pi^2}|\hat u_1 (\xi)|^2 + \dfrac{1}{4}| \hat u_0 (\xi)|^2 \right) e^{-\varphi(\xi)t} d\xi\\
&+ 18\int_{1<|\xi|\leqslant k}\left( \dfrac{1}{\log ^2(1 + |\xi|^2)+\pi^2}| \hat u_1 (\xi)|^2 + \dfrac{1}{4}|\hat u_0 (\xi)|^2 \right) e^{-\varphi(\xi)t} d\xi\\
&+ 18\int_{|\xi|> k}\left( \dfrac{1}{\log ^2(1 + |\xi|^2)+\pi^2}| \hat u_1 (\xi)|^2 + \dfrac{1}{4}| \hat u_0 (\xi)|^2 \right) e^{-\varphi(\xi)t} d \xi \\
&=:18(B_1+B_2+B_3),
\end{split}
\]
with $B_i$ $ (i = 1, 2, 3)$ according to the integrals on low, middle and high frequencies, respectively.

On the low frequency zone $|\xi|\leqslant 1$ we have $\varphi (\xi) =\frac23 \log (1+|\xi |^2)$, and it holds the estimates

\[
\begin{split} 
B_1 &= \int_{|\xi|\leqslant 1} \dfrac{1}{\log ^2(1 + |\xi|^2)+\pi^2}| \hat u_1 (\xi)|^2 e^{-\frac{2\log(1+|\xi|^2)}{3}t} d\xi  + \dfrac{1}{4}\int_{|\xi|\leqslant 1}|\hat u_0 (\xi)|^2  e^{-\frac{2\log(1+|\xi|^2)}{3}t} d\xi \\
&\leqslant \|\hat{u}_1\|_\infty^2\int_{|\xi|\leqslant 1}\dfrac{1}{\log^2 (1+|\xi|^2)+\pi^2}(1+|\xi|^2)^{-\frac{2}{3}t}d\xi + \dfrac{\|\hat{u}_0\|_\infty^2}{4}\int_{|\xi|\leqslant 1}(1+|\xi|^2)^{-\frac{2}{3}t}d\xi\\
&\leqslant \dfrac{\|\hat{u}_1\|_\infty^2 }{\pi^2}\int_{|\xi|\leqslant 1}(1+|\xi|^2)^{-\frac{2}{3}t}d\xi + \dfrac{\|\hat{u}_0\|_\infty^2}{4}\int_{|\xi|\leqslant 1}(1+|\xi|^2)^{-\frac{2}{3}t}d\xi\\
&\leqslant \dfrac{\|u_1\|_{1}^{2} }{\pi^2}\omega_N \int_{0}^{1} (1+r^2)^{-\frac{2}{3}t} r^{N-1} dr + \dfrac{\|u_0\|_{1}^{2}}{4} \omega_N\int_{0}^{1} (1+r^2)^{-\frac{2}{3}t}r^{N-1}dr,
\end{split} 
\]
where $\omega_N := \int_{|\omega| = 1} d\omega$ is the surface area of the $N$-dimensional unit ball.

Thus, for $t \gg 1$, by Lemma \ref{l2.1}, we have

\[
\begin{split} 
B_1 &\leqslant \dfrac{\|u_1\|_1^2 }{\pi^2}\omega_N I_{N-1}\left(\frac{2}{3}t\right) + \dfrac{\|u_0\|_{1}^{2}}{4} \omega_N I_{N-1}\left(\frac{2}{3}t\right)\\
&= \dfrac{\|u_1\|_1^2 }{\pi^2}\omega_N \left(\frac{2}{3}t\right)^{-\frac{N}{2}} + \dfrac{\|u_0\|_{1}^{2}}{4} \omega_N \left(\frac{2}{3}t\right)^{-\frac{N}{2}}\\
&\leqslant C_N t^{-\frac{N}{2}} \left( \|u_1\|_1^2  + \|u_0\|_1^2 \right),
\end{split} 
\]
where $C_N$ is a constant depending only on $N$.

On the middle frequency zone, we have $1< |\xi| \leqslant \sqrt{e^{\frac43}-1}$, and it holds the following inequality
\begin{equation*}
\log (2) \leqslant \log(1+|\xi|^2).
\end{equation*}Thus, one has for all $t>0$,

\[
\begin{split} 
B_2 &= \int_{1< |\xi|\leqslant k} \dfrac{1}{\log ^2(1 + |\xi|^2)+\pi^2}|\hat u_1 (\xi)|^2 e^{-\frac{2\log(1+|\xi|^2)}{3}t} d\xi  + \dfrac{1}{4}\int_{1\leqslant |\xi|\leqslant k}|\hat u_0 (\xi)|^2  e^{-\frac{2\log(1+|\xi|^2)}{3}t} d\xi \\
&\leqslant \dfrac{1}{\pi^2} \int_{1< |\xi|\leqslant k}|\hat u_1 (\xi)|^2 e^{-\frac23 (\log 2) t} d\xi+ \dfrac{1}{4}\int_{1\leqslant |\xi|\leqslant k}|\hat u_0 (\xi)|^2  e^{-\frac23 (\log 2) t} d\xi \\
&= \dfrac{1}{\pi^2} \int_{1< |\xi|\leqslant k}|\hat u_1 (\xi)|^2 2^{-\frac23 t} d\xi+ \dfrac{1}{4}\int_{1\leqslant |\xi|\leqslant k}|\hat u_0 (\xi)|^2 2^{-\frac23 t} d\xi \\
&\leqslant \left[\dfrac{1}{\pi^2}\Vert u_1\Vert_{2}^2+\dfrac{1}{4}\Vert u_0\Vert_{2}^2 \right]2^{-\frac23 t}.
\end{split} 
\]

Finally, we analyzing the estimates on the high frequency zone. On this region, we have $\varphi (\xi)=\dfrac{8}{9}$. Thus, we obtain the estimate 

\[
\begin{split} 
B_3 &= \int_{|\xi|> k} \dfrac{1}{\log ^2(1 + |\xi|^2)+\pi^2}|\hat u_1 (\xi)|^2 e^{-\frac89 t} d\xi  + \dfrac{1}{4}\int_{|\xi|> k}|\hat u_0 (\xi)|^2  e^{-\frac89 t} d\xi \\
&\leqslant \dfrac{1}{\pi^2} \int_{|\xi|> k}|\hat u_1 (\xi)|^2 e^{-\frac89 t} d\xi+ \dfrac{1}{4}\int_{|\xi|> k}|\hat u_0 (\xi)|^2  e^{-\frac89 t} d\xi \\
&\leqslant \left[\dfrac{1}{\pi^2}\Vert u_1\Vert_{2}^2+\dfrac{1}{4}\Vert u_0\Vert_{2}^2 \right]e^{-\frac89 t},
\end{split} 
\]for $t>0$.
\qed
\end{proof}

\section{Asymptotic profile}\label{FinalSection}

To obtain an asymptotic profile of solutions, without  loss of generality, we consider the case where the initial condition $u_0 = 0$. With this, the Cauchy problem \eqref{Pro1}-\eqref{Pro2} in the Fourier space is given by
\begin{equation}\label{Pro5}
\partial_t^2 \hat u (\xi,t)+\dfrac{1}{4}\log^2(1+|\xi|^2)\hat u (\xi,t)+\dfrac{\pi^2}{4}\hat u (\xi,t)+\log(1+|\xi|^2)\partial_t \hat u(\xi,t)=0,
\end{equation}
for $ \xi\in\mathbb{R}^N,\ t>0$, subject to initial conditions
\begin{equation}\label{Pro6}
\hat u(\xi,0)= 0,\ \partial_t \hat u(\xi,0)=\hat u_1 (\xi),\ \xi\in\mathbb{R}^N,
\end{equation}

The solution formula can be expressed by 
\begin{equation}\label{sform}
    \hat u (\xi,t) = e^{-\dfrac{t\log(1+|\xi|^2)}{2}}\dfrac{4}{\pi}\sin\Big(\dfrac{\pi}{4}t\Big)\hat u_1 (\xi),
\end{equation}
for  $\xi \in \mathbb{R}^N$ and $t>0$.

\begin{remark}\label{decomp}
Using the Fourier transform, we can obtain a decomposition of the initial condition $\hat u_1$ as follows:
\[
\hat u_1 (\xi) = A(\xi) - iB(\xi) + P_1, \ \xi \in \mathbb{R}^N,
\]
where
    \[
    P_1 = \int_{\mathbb{R}^N}u_1(x)dx, 
    \]
    \[
    A(\xi) = \int_{\mathbb{R}^N}u_1(x)(\cos(\xi x) - 1)dx,
    \]
and
\[
B(\xi) = \int_{\mathbb{R}^N}u_1(x)\sin(\xi x)dx.
\]
\end{remark}

According to the above decomposition, we can state the following lemma (see \cite[Lemma 4.1]{CIP}).

\begin{lemma}\label{est1}
    Let $\kappa \in [0,1]$. For each $u_1 \in L^{1, \kappa}(\mathbb{R}^N)$ and $\xi \in \mathbb{R}^N$, we have
\[
    |A(\xi)| \leqslant K|\xi|^{\kappa}\|u_1\|_{L^{1,\kappa}}
\]
and
\[
|B(\xi)| \leqslant M|\xi|^{\kappa}\|u_1\|_{L^{1,\kappa}}, 
    \]
with positive constants $K$ and $M$ that depend only on $N$.    
\end{lemma}

First, we apply the mean value theorem, to obtain 

\[
\sin\Big(\dfrac{\pi}{4}t\Big) - \sin\Big(t\sqrt{\log(1 + |\xi|^2)}\Big) = t\cos(\mu (\xi)t)\Big[\dfrac{\pi}{4} - \sqrt{\log(1 + |\xi|^2)}\Big],
\]
where  $\mu(\xi) = \dfrac{\pi}{4}\theta + (1 - \theta)\sqrt{\log(1 + |\xi|^2)}$, $\theta \in (0,1)$.

Therefore, we can rewrite the solution formula \eqref{sform} as
\[
\begin{split}
    \hat u (\xi,t) &= \dfrac{4}{\pi}(A(\xi) - iB(\xi))e^{-\frac{t\log(1+|\xi|^2)}{2}}\sin\Big(\dfrac{\pi}{4}t\Big)\\
    &+  \dfrac{4}{\pi}P_1 e^{-\frac{t\log(1+|\xi|^2)}{2}} \sin\Big(t\sqrt{\log(1 + |\xi|^2)}\Big)\\
    &+ \dfrac{4}{\pi}P_1\Big[\dfrac{\pi}{4} - \sqrt{\log(1 + |\xi|^2)}\Big] e^{-\frac{t\log(1+|\xi|^2)}{2}}t \cos(\mu(\xi)t). 
\end{split}
\]
In this section, we obtain decay estimates in time for the remaining terms defined above. First, we define the following  two functions
\begin{equation}\label{FgDefF3}
\begin{split}
    &F_1(\xi, t) = \dfrac{4}{\pi}(A(\xi) - iB(\xi))e^{-\frac{t\log(1+|\xi|^2)}{2}}\sin\Big(\dfrac{\pi}{4}t\Big)\\
    &F_2(\xi, t) = \dfrac{4}{\pi}P_1\Big[\dfrac{\pi}{4} - \sqrt{\log(1 + |\xi|^2)}\Big] e^{-\frac{t\log(1+|\xi|^2)}{2}}t\cos(\mu(\xi)t) \\\
    & F_3(\xi, t) = \dfrac{4}{\pi}P_1 e^{-\frac{t\log(1+|\xi|^2)}{2}} \sin\Big(t\sqrt{\log(1 + |\xi|^2)}\Big).
\end{split}
\end{equation}
Then, we obtain 
\[
\hat u (\xi, t) - F_3(\xi, t) = F_1(\xi, t) + F_2(\xi, t).
\]
Consider $\kappa = 1$. Then using remark \ref{decomp} and Lemma \ref{est1}, we obtain 
\[
\begin{split}
    \int_{|\xi| < 1} |F_1(\xi, t)|^2 d\xi &=\dfrac{16}{\pi^2}\int_{|\xi| < 1}|A(\xi) - iB(\xi)|^2\sin^2\Big(\dfrac{\pi}{4}t\Big)e^{-t\log(1+|\xi|^2)} d\xi\\
    & \leqslant \dfrac{16}{\pi^2}\int_{|\xi| < 1}\left(|A(\xi)| + |B(\xi)|\right)^2 e^{-t\log(1+|\xi|^2)}d\xi\\
    & \leqslant \dfrac{16}{\pi^2}\int_{|\xi| < 1}\left(K + M\right)^2\|u_1\|_{1,1}^2 e^{-t\log(1+|\xi|^2)}d\xi\\
    & \leqslant \dfrac{16}{\pi^2}\omega_N \left(K + M\right)^2\|u_1\|_{1,1}^2\int_0^1 r^{N + 1}(1 + r^2)^{-t} dr\\
    & = \dfrac{16}{\pi^2}\omega_N \left(K + M\right)^2\|u_1\|_{1,1}^2 I_{N + 1}(t)\\
    &\leqslant C_{1,N}\|u_1\|_{1,1}^2 t^{-\frac{N+2}{2}},
\end{split}
\]
where $\omega_N := \int_{|\omega| = 1} d\omega$ is the surface area of the $N$-dimensional unit ball.

On the other hand, as
\[
\dfrac{\pi}{4} - \sqrt{\log(1 + |\xi|^2)} \leqslant \dfrac{\pi}{4},
\]
and from Lemma \ref{l2.1}, we have
\[
\begin{split}
    \int_{|\xi| < 1} |F_2(\xi, t)|^2 d\xi 
    &= \dfrac{16}{\pi^2} |P_1|^2 t^2 \int_{|\xi| < 1} \Big|\dfrac{\pi}{4} - \sqrt{\log(1 + |\xi|^2)}\Big|^2 e^{-t\log(1+|\xi|^2)}\cos^2(\mu(\xi)t)  d\xi \\
    & \leqslant \dfrac{16}{\pi^2}|P_1|^2 t^2  \int_{|\xi| < 1} \dfrac{\pi^2}{16}(1+|\xi|^2)^{-t} d\xi \\ 
    &\leqslant \omega_N |P_1|^2 t^2 \int_0^1 r^2 (1+r^2)^{-t} r^{N - 3}dr\\
    &\leqslant \omega_N |P_1|^2 t^2 \int_0^1 (1+r^2)^{-t} r^{N - 3}dr\\
    &=  \omega_N |P_1|^2 t^2  I_{N - 3}(t)\\
    &\leqslant C_{2,N}|P_1|^2t^{-\frac{N - 2}{2}}.
\end{split}
\]
\begin{theorem} \label{t31}
    Let $u_1 \in L^2(\mathbb{R}^N) \cap L^{1,1}(\mathbb{R}^N)$. Then

\[
\int_{|\xi| < 1}|\hat u(\xi, t) - F_3(\xi, t)|^2 d\xi \leqslant C_{1,N}\|u_1\|_{1,1}^2t^{-\frac{N+2}{2}} + C_{2,N}|P_1|^2t^{-\frac{N - 2}{2}}
\]
for $t\gg 1$ with positive constants $C_{1,N}$ and $C_{2,N}$ that depend only on $N \in \mathbb{N}$.
\end{theorem}

\begin{theorem} \label{t32}
    Let $u_1 \in L^2(\mathbb{R}^N) \cap L^{1}(\mathbb{R}^N) $. Then, 
    \[
        \int_{|\xi| \geqslant 1} |\hat u (\xi, t) - F_3(\xi, t)|^2 d\xi \leqslant C(\|u_1\|^2_2 + |P_1|^2)e^{-\eta t},
   \]
for $t \gg 1$ with positive constants $C$ and $\eta$ that depend on $N \in \mathbb{N}$.
\end{theorem}

\begin{proof}
    From Proposition \ref{eq2prop3.1} it follows that 
    \begin{equation}\label{eq3.4}
        \begin{split}
            |\hat u (\xi,t)|^2 &\leqslant \dfrac{18}{\log ^2(1 + |\xi|^2)+\pi^2}|\hat u_1 (\xi)|^2 e^{-\frac{8}{9}t}\\
            &\leqslant \dfrac{18}{\frac{16}{9}+\pi^2}|\hat u_1 (\xi)|^2 e^{-\frac{8}{9}t}\\
            &\leqslant \dfrac{162}{16+9\pi^2}|\hat u_1 (\xi)|^2 e^{-\frac{8}{9}t}\\
        \end{split}
    \end{equation}
for $|\xi| \geqslant \sqrt{e^{\frac{4}{3}} - 1} $.
In addition,
\begin{equation}\label{eq3.5}
    |F_3(\xi, t)|^2 \leqslant \dfrac{16}{\pi^2}|P_1|^2 (1+|\xi|^2)^{-t},
\end{equation}
for $|\xi| \geqslant \sqrt{e^{\frac{4}{3}} - 1}$.
Thus, from \eqref{eq3.4}, \eqref{eq3.5} and Lemma \ref{l2.2}, we have
\begin{equation}\label{eq3.6}
    \begin{split}
        \int_{|\xi| \geqslant \sqrt{e^{\frac{4}{3}} - 1}} |\hat u(\xi, t) - F_3(\xi, t)|^2 d\xi & \leqslant 2\int_{|\xi| \geqslant \sqrt{e^{\frac{4}{3}} - 1}} \left(|\hat u(\xi, t)|^2 +|F_3(\xi, t)|^2 \right) d\xi\\
        &\leqslant \dfrac{324}{16+9\pi^2}e^{-\frac{8}{9}t}\| u_1\|^2_2  + \dfrac{32}{\pi^2}|P_1|^2\omega_N\int_{1}^{\infty}(1 + r^2)^{-t}r^{N - 1}dr\\
        &= \dfrac{324}{16+9\pi^2}e^{-\frac{8}{9}t}\| u_1\|^2_2  + \dfrac{32}{\pi^2}|P_1|^2\omega_NJ_{N - 1}(t)\\
        &\leqslant \dfrac{324}{16+9\pi^2}e^{-\frac{8}{9}t}\| u_1\|^2_2  + \dfrac{32}{\pi^2}|P_1|^2\omega_N\dfrac{2^{-t}}{t - 1},
\end{split}
\end{equation}
for $t \gg 1$.

Now, by Proposition \ref{Prop3.1Refdasrjhk} it follows that 
\begin{equation}\label{eq3.7}
        \begin{split}
            |\hat u (\xi,t)|^2 &\leqslant \dfrac{18}{\log ^2(1 + |\xi|^2)+\pi^2}|\hat u_1 (\xi)|^2 e^{-\frac{2}{3}\log(1+|\xi|^{2})t}\\
            &\leqslant \dfrac{18}{\log ^2(2)+\pi^2}|\hat u_1 (\xi)|^2 e^{-\frac{2}{3}\log(2)t},
        \end{split}
    \end{equation}
for $1 \leqslant |\xi| \leqslant \sqrt{e^{\frac{4}{3}} - 1}$.
We also have
\begin{equation}\label{eq3.8}
        \begin{split}
            |F_3(\xi, t)|^2 &\leqslant \dfrac{16}{\pi^2}|P_1|^2 e^{-\log(1+|\xi|^{2})t}\\
            &\leqslant \dfrac{16}{\pi^2}|P_1|^2 2^{-t},
        \end{split}
    \end{equation}
for $1 \leqslant |\xi| \leqslant \sqrt{e^{\frac{4}{3}} - 1}$.
Finally, thanks to estimates \eqref{eq3.6}, \eqref{eq3.7}, and \eqref{eq3.8} in the high and middle frequency zones $|\xi| \geqslant \sqrt{e^{\frac{4}{3}} - 1}$ and $1 \leqslant |\xi| \leqslant \sqrt{e^{\frac{4}{3}} - 1}$ the result follows.
\qed
\end{proof}

\begin{theorem}
    Let $u_0 = 0$ and $u_1 \in \left(L^2(\mathbb{R}^N) \cap L^{1,1}(\mathbb{R}^N) \right)$. Then, the unique solution $u=u(x,t)$ to problem \eqref{Pro1}-\eqref{Pro2} satisfies
    \[
    \left\|u(\cdot, t) - \left(\int_{\mathbb{R}^N }u_1(x) dx\right)\mathcal{F}^{-1}\left(\dfrac{4(1+|\xi|^2)^{-\frac{t}{2}}}{\pi}  \sin\Big(t\sqrt{\log(1 + |\xi|^2)}\Big)\right)\right\|_{2} \leqslant I_0t^{-\frac{N-2}{4}}
    \]
for $t \gg 1$, with 
\[
I_0 = \|u_1\|_{2} + \|u_1\|_{1,1}.
\]

\begin{proof}
    This result is a direct consequence of theorems \ref{t31}, \ref{t32}, and the Plancherel's theorem, see e.g. \cite[Theorem 1, p. 183]{Evans}.
    \qed
\end{proof}
\end{theorem}

\section{Optimal decay rate of $L^2-$norm}\label{sec4}

Finally, we define the following function
\begin{equation}\label{TeTTe}
\phi(\xi,t)=P_1e^{-\frac{t\log(1+|\xi|^2)}{2}} \sin\Big(t\sqrt{\log(1 + |\xi|^2)}\Big)
\end{equation}
which is equivalent to $F_3$ given in \eqref{FgDefF3} for $|\xi|\leqslant1$. In this section, using the ideas of \cite[Section 5]{CIP}, we investigate the precise rate of decay of the leading term \eqref{TeTTe} in $L^2-$sense as $t$ goes to infinity.

\begin{theorem}
Let $N\geqslant3$. Then there exists $t_0>0$ such that for $t\geqslant t_0$ it holds that
\[
C_{1,N}t^{-\frac{N}{2}}\leqslant\int_{\mathbb{R}^N} (1 + |\xi|^2)^{-t}\sin^2\Big(t\sqrt{\log(1 + |\xi|^2)}\Big)d\xi\leqslant C_{2,N}t^{-\frac{N}{2}},
\] 
where $C_{1,N}$ and $C_{2,N}$ are positive constants depending  only on $N$.
\end{theorem}

\begin{proof}
Firstly, we observe that  
\[
\begin{split}
\int_0^\infty (1 + r^2)^{-t}\sin^2\Big(t\sqrt{\log(1 + r^2)}\Big) r^{N-1}dr&=\int_0^1 (1 + r^2)^{-t}\sin^2\Big(t\sqrt{\log(1 + r^2)}\Big) r^{N-1}dr\\ 
&=\int_1^\infty (1 + r^2)^{-t}\sin^2\Big(t\sqrt{\log(1 + r^2)}\Big) r^{N-1}dr.
\end{split}
\]

 Thus, for $t \gg 1$ based on lemmas \ref{l2.1} and \ref{l2.2} we get
\[
\begin{split}
\int_0^1 (1 + r^2)^{-t}\sin^2\Big(t\sqrt{\log(1 + r^2)}\Big) r^{N-1}dr&=\int_0^1 (1 + r^2)^{-t}  r^{N-1}dr\\
&\leqslant C_{1,N}t^{-\frac{N}{2}},
\end{split}
\]
and
\[
\begin{split}
\int_1^\infty (1 + r^2)^{-t}\sin^2\Big(t\sqrt{\log(1 + r^2)}\Big) r^{N-1}dr&\leqslant \int_1^\infty (1 + r^2)^{-t} r^{N-1}dr\\
&\leqslant C_{2,N}\dfrac{2^{-t}}{t-1}.
\end{split}
\]

These two estimates above imply that there exists $t_0>0$ such that 
\[
\int_1^\infty (1 + r^2)^{-t}\sin^2\Big(t\sqrt{\log(1 + r^2)}\Big) r^{N-1}dr\leqslant C_{3,N}t^{-\frac{N}{2}}
\]
for all $t\geqslant t_0$.

On the other hand, we know that $r^ 2\geqslant\log(1+r^2)$, then  
\[
\begin{split}
&\int_0^\infty (1 + r^2)^{-t}\sin^2\Big(t\sqrt{\log(1 + r^2)}\Big) r^{N-1}dr\\
&= \int_0^\infty \dfrac{\log(1+r^2)(1 + r^2)^{-t}\sin^2\Big(t\sqrt{\log(1 + r^2)}\Big) r^{N-1}}{\log(1+r^2)}dr\\
&\geqslant \int_0^\infty \dfrac{(1 + r^2)^{-t}\sin^2\Big(t\sqrt{\log(1 + r^2)}\Big)\log(1+r^2) r^{N-2}\sqrt{t} r}{\sqrt{t}\sqrt{\log(1+r^2)}(1+r^2)\sqrt{\log(1+r^2)}}dr\\
&\geqslant \int_0^\infty \dfrac{(1 + r^2)^{-t}\sin^2\Big(t\sqrt{\log(1 + r^2)}\Big)\log^{\frac{N}{2}}(1+r^2)\sqrt{t} r}{\sqrt{t}\sqrt{\log(1+r^2)}(1+r^2)\sqrt{\log(1+r^2)}}dr.
\end{split}
\]
By using the change of variable $y=\sqrt{t}\sqrt{\log(1+r^2)}$ we have
\[
\begin{split}
&\int_0^\infty (1 + r^2)^{-t}\sin^2\Big(t\sqrt{\log(1 + r^2)}\Big) r^{N-1}dr\\
&\geqslant \int_0^\infty\dfrac{e^{-y^2}\sin^2(\sqrt{t}y)y^{N}}{t^{\frac{N}{2}}y}dy\\
&\geqslant \int_0^\infty\dfrac{e^{-y^2}\sin^2(\sqrt{t}y)y^{N-1}}{t^{\frac{N}{2}}}dy\\
&= t^{-\frac{N}{2}}\int_0^\infty e^{-y^2} y^{N-1}dy -t^{-\frac{N}{2}}\int_0^\infty e^{-y^2}\cos^2(\sqrt{t}y)y^{N-1}dy\\
&= t^{-\frac{N}{2}}(A_N-F_N(t)),   
\end{split}
\]
where
\[
A_N=\int_0^\infty e^{-y^2}y^{N-1}dy,
\]
and
\[
F_N(t)=\int_0^\infty e^{-y^2}\cos^2(\sqrt{t}y)y^{N-1}dy.
\]
Due to the fact $e^{-y^2}y^{N-1}\in L^1(\mathbb{R})$ for $N\geqslant 3$, we can apply the Riemann-Lebesgue Lemma to get
\[
\lim_{t\to+\infty}F_N(t)=0.
\]
Then there exists $t_1>t_0$ such that $F_N(t)\leqslant \dfrac{A_N}{2}$ for any $t\geqslant t_1$; that is,
\[
A_N-F_N(t)\geqslant\dfrac{A_N}{2}t^{-\frac{N}{2}}
\]
for any $t\geqslant t_1$.

Thus,
\[
\int_0^\infty (1 + r^2)^{-t}\sin^2\Big(t\sqrt{\log(1 + r^2)}\Big) r^{N-1}dr\geqslant \dfrac{A_N}{4}t^{-\frac{N}{2}}.
\]
for any $t\geqslant t_1$.
\end{proof}

\end{document}